\pdfoutput=1
\RequirePackage{ifpdf}
\ifpdf 
\documentclass[pdftex]{sigma}
\else
\documentclass{sigma}
\fi

\numberwithin{equation}{section}

\newtheorem{Theorem}{Theorem}[section]
\newtheorem{Lemma}[Theorem]{Lemma}
\newtheorem{Proposition}[Theorem]{Proposition}
 { \theoremstyle{definition}
\newtheorem{Definition}[Theorem]{Definition}
\newtheorem{Remark}[Theorem]{Remark} }

\begin{document}
\allowdisplaybreaks

\newcommand{\arXivNumber}{2110.07042}

\renewcommand{\PaperNumber}{113}

\FirstPageHeading

\ShortArticleName{Orthogonal Polynomial Stochastic Duality Functions for Multi-Species SEP$(2j)$}

\ArticleName{Orthogonal Polynomial Stochastic Duality Functions\\ for Multi-Species SEP$\boldsymbol{(2j)}$ and Multi-Species IRW}

\Author{Zhengye ZHOU}

\AuthorNameForHeading{Z.~Zhou}

\Address{Department of Mathematics, Texas A$\&$M University, College Station, TX 77840, USA}
\Email{\href{mailto:zyzhou@tamu.edu}{zyzhou@tamu.edu}}

\ArticleDates{Received October 16, 2021, in final form December 24, 2021; Published online December 26, 2021}

\Abstract{We obtain orthogonal polynomial self-duality functions for multi-species version of the symmetric exclusion process (SEP$(2j)$) and the independent random walker process (IRW) on a finite undirected graph. In each process, we have $n>1$ species of particles. In addition, we allow up to $2j$ particles to occupy each site in the multi-species SEP$(2j)$. The duality functions for the multi-species SEP$(2j)$ and the multi-species IRW come from unitary intertwiners between different $*$-representations of the special linear Lie algebra $\mathfrak{sl}_{n+1}$ and the Heisenberg Lie algebra $\mathfrak{h}_n$, respectively. The analysis leads to multivariate Krawtchouk polynomials as orthogonal duality functions for the multi-species SEP$(2j)$ and homogeneous products of Charlier polynomials as orthogonal duality functions for the multi-species IRW.}

\Keywords{orthogonal duality; multi-species SEP$(2j)$; multi-species IRW}

\Classification{60K35}

\section{Introduction}
In recent years, stochastic duality has been used as a powerful tool in the study of stochastic processes (see, e.g., \cite{Borodin_2014,corwin2016asepqj,kuan2019stochastic,zhou2020hydrodynamic}). More recently, orthogonal stochastic dualities were derived for some classical interacting particle systems. For instance, the independent random walker process (IRW) is self-dual with respect to Charlier polynomials \cite{Carinci_2019,Groenevelt_2018}, and the symmetric exclusion process (SEP$(2j)$) is self-dual with respect to Krawtchouk polynomials \cite{Carinci_2019,Groenevelt_2018}. Orthogonal polynomial duality functions turn out to be useful in applications as they form a convenient orthogonal basis in a suitable space of the systems’ observables. There are many applications of orthogonal duality functions (see, e.g., \cite{Ayala_2018,ayala2020higher,chen2021higher}). In a series of previous works, several ways to find orthogonal dualities were introduced. In \cite{franceschini2017stochastic}, the two-terms recurrence relations of orthogonal polynomials were used, method via generating functions was used in \cite{2018}, while Lie algebra representations and unitary intertwiners were used in~\cite{Groenevelt_2018}. In addition, two more approaches were described in \cite{Carinci_2019}. The first approach is based on unitary symmetries, another one is based on scalar products of classical duality functions. In this paper, we make use of the method introduced in~\cite{Groenevelt_2018}.

We first study a model of interacting particle systems with symmetric jump rates. The multi-species symmetric exclusion process is a generalization of the SEP$(2j)$ to multi-species systems, where we have up to $2j\in\mathbb{N}$ particles allowed for each site and we have $n>1$ species particles in the system. It is worth mentioning that the multi-species SEP$(2j)$ we consider is closely related to other multi-species (multi-color) exclusion processes studied over the past decades. For example, when $j=\frac{1}{2}$, it degenerates to a special case of the multi-color exclusion
processes studied in \cite{Caputo2008ONTS,Dermoune2008SpectralGF}. This model also arises naturally as a special case of the multi-species ASEP$(q,j)$ studied in \cite{Kuan_2017} when $q=1$. Given the fact that the single-species SEP$(2j)$ is self-dual with respect to Krawtchouk polynomials, it's expected that similar results could be found for the multi-species SEP$(2j)$. We prove that the multi-species SEP$(2j)$ is self-dual with multivariate Krawtchouk polynomials as duality functions.

Another process we study is the multi-species independent random walker, which can be thought of as $n>1$ independent copies of IRW evolving simultaneously. Although it is straightforward to obtain the duality functions using the independence property, it is still interesting to show how the duality functions arise from representations of the nilpotent Heisenberg Lie algebra $\mathfrak{h}_n$.

The organization of this paper is as follows. In Section~\ref{section2} we give an overview of the method that we use to construct orthogonal duality functions. In Section~\ref{section3} we obtain the orthogonal duality functions for the multi-species SEP$(2j)$ and in Section~\ref{section4} for the multi-species IRW.

\section{Background}\label{section2}
In this section we describe the method to obtain the orthogonal dualities which was introduced in \cite{Groenevelt_2018}. We start by recalling the definition of stochastic duality.

\begin{Definition}
Two Markov processes $\mathfrak{s}_t$ and $\mathfrak{s}_t'$ on state spaces $\mathfrak{S}$ and $\mathfrak{S}'$ are dual with respect to duality function $D(\cdot,\cdot)$ on $\mathfrak{S}\times\mathfrak{S}'$ if
\begin{gather*}
 E_{\mathfrak{s}}[D(\mathfrak{s}_t,\mathfrak{s}')]=E'_{\mathfrak{s}'}[D(\mathfrak{s},\mathfrak{s}_t')] \qquad \text{for all} \ \mathfrak{s} \in \mathfrak{S},\ \mathfrak{s}'\in \mathfrak{S}', \ \text{and} \ t>0,
\end{gather*}
where $E_{\mathfrak{s}}$ denotes expectation with respect to the law of $\mathfrak{s}_t$ with $\mathfrak{s}_0=\mathfrak{s}$ and similarly for $E'_{\mathfrak{s}'}$. If~$\mathfrak{s}_t'$ is a copy of $\mathfrak{s}_t$, we say that the process $\mathfrak{s}_t$ is self-dual.

In most relevant examples, duality could also be stated at the level of
Markov generators. We say that generator $L_1$ is dual to $L_2$ with respect to duality function $D(\cdot,\cdot)$ if for all~$\mathfrak{s}$ and~$\mathfrak{s}'$,
\begin{gather*}
 [L_1D(\cdot,\mathfrak{s}')](\mathfrak{s})=[L_2D(\mathfrak{s},\cdot)](\mathfrak{s}').
\end{gather*}
If $L_1=L_2$, we have self-duality.
\end{Definition}

Let $\mathfrak{g}=(\mathfrak{g},[\cdot,\cdot])$ be a complex Lie algebra with a $*$-structure, i.e., there exists an involution $*\colon X\xrightarrow[]{}X^*$ such that for any $X,Y\in \mathfrak{g}$, $a,b\in \mathbb{C}$,
\begin{gather*}
 (aX+bY)^*=\overline{a}X^*+\overline{b}Y^*,\qquad [X,Y]^*=[Y^*,X^*].
\end{gather*}
Let $U(\mathfrak{g})$ be the universal enveloping algebra of~$\mathfrak{g}$.

Given a Hilbert space $(H,\langle \cdot,\cdot\rangle)$
and a representation $\rho$ of $\mathfrak{g}$ on $H$, we call $\rho$ a $*$-representation if for any $f,g\in H$ and any $X\in\mathfrak{g} $,
\begin{gather*}
\langle \rho(X)f,g\rangle =\langle f,\rho(X^*)g\rangle .
\end{gather*}
Suppose we have state spaces $\Omega_1$ and $\Omega_2$ of configurations on $L$ sites given by $\Omega_1=E_1\times\dots\times E_L$ and $\Omega_2=F_1\times\dots\times F_L$. Let $\mu=\mu_1\otimes\dots\otimes \mu_L$ and $\nu=\nu_1\otimes\dots\otimes \nu_L$ be product measures on~$\Omega_1$ and~$\Omega_2$.

For $1\le x\le L$, let $\rho_x$ and $\sigma_x$ be unitarily equivalent $*$-representations of a Lie algebra $\mathfrak{g}$ on $ L^2(E_x,\mu_x)$ and $ L^2(F_x,\nu_x)$, respectively. Then $\rho=\rho_1\otimes\dots\otimes \rho_L$ and $\sigma=\sigma_1\otimes\dots\otimes\sigma_L$ are $*$-representations of $\mathfrak{g}$. We assume that the corresponding unitary intertwiner $\Lambda_x\colon L^2(E_x,\mu_x)\xrightarrow{}L^2(F_x,\nu_x)$ has the following form:
\begin{gather*}
 (\Lambda_x f)(z_2)=\int_{E_x} f(z_1)K_x(z_1,z_2)\,{\rm d}\mu_x(z_1),\qquad \text{for $\nu_x$-almost all $z_2\in F_x$},
\end{gather*}
for some kernel $K_x\in L^2(E_x\times F_x,\mu_x\otimes \nu_x)$ satisfying the relation
\begin{gather*}
 [\rho_x(X^*)K_x(\cdot,z_2)](z_1)=[\sigma_x(X)K_x(z_1,\cdot)](z_2), \qquad (z_1,z_2)\in E_x\times F_x,\quad X\in\mathfrak{g}.
\end{gather*}

With all the above structures, the following theorem provides a way to construct duality functions.
\begin{Theorem}[\cite{Groenevelt_2018}]\label{th:2.1}
Suppose $L_1$ and $L_2$ are self-adjoint operators on $L^2(\Omega_1,\mu)$ and $L^2(\Omega_2,\nu)$, respectively, given by
\begin{gather*}
 L_1=\rho(Y),\qquad L_2=\sigma(Y),
\end{gather*}
for some self-adjoint $Y\in U(\mathfrak{g})^{\otimes L}$. Then $L_1$ and $L_2$ are in duality, with duality function
\begin{gather*}
 D(z_1,z_2)=\prod_{x=1}^L K_x(z_{1x},z_{2x}), \qquad z_1=(z_{11},\dots,z_{1L})\in \Omega_1,\qquad z_2=(z_{21},\dots,z_{2L})\in \Omega_2.
\end{gather*}
\end{Theorem}

\section[Multi-species SEP(2j) and Lie algebra sl\_\{n+1\}]{Multi-species SEP$\boldsymbol{(2j)}$ and Lie algebra $\boldsymbol{\mathfrak{sl}_{n+1}}$}\label{section3}

In this section, we study the multi-species version of the SEP$(2j)$ with $2j\in\mathbb{N}$ on a finite undirected graph $G=(V,E)$, where $V=\{1,\dots,L\}$ with $L\in\mathbb{N}$ and $L>2$ is the set of sites (vertices) and $E$ is the set of edges. In what follows, we write site $x\in G$ instead of mentioning~$V$ for ease of notation.

The state space $\mathcal{S}(n,2j,G)$ of particle configurations consists of variables $\xi=\big(\xi_i^x\colon 0\le i\le n$, $x\in G\big)$, where $\xi_i^x$ denotes the number of particles of species $i$ at site $x$, and
\begin{gather*}
 \xi^x=\big(\xi^x_0,\dots,\xi^x_n\big)\in \Omega_{2j}:=\left\{\xi=(\xi_0,\dots,\xi_n)\,\Bigg|\,\sum_{i=0}^n \xi_i=2j,\, \xi_i\ge 0\right\}
\end{gather*}
for any site $x\in G $. One can think of $\xi^x_0$ as the number of holes at site~$x$.

\begin{Definition}\label{def:3.1}
The generator of the multi-species SEP$(2j)$ on a finite undirected graph $G=(V,E)$ is given by
\begin{gather}
\begin{split}
 &\mathcal{L}f(\xi)=\sum_{\text{edge}\{x,y\}\in E} \mathcal{L}_{x,y} f(\xi), \\
& \mathcal{L}_{x,y}f(\xi)=\sum_{0\le k<l\le n}\xi_l^{x}\xi_k^{y}\big[f\big(\xi^{x,y}_{l,k}\big)-f(\xi)\big]+\xi_l^{y}\xi^{x}_k\big[f\big(\xi^{y,x}_{l,k}\big)-f(\xi)\big],
\end{split}\label{eq:6}
\end{gather}
where $\xi_{l,k}^{x,y}$ denotes the particle configuration obtained by switching a particle of the $l^{\rm th}$ species at site $x$ with a particle of the $k^{\rm th}$ species at site $y$ if $\xi_{l,k}^{x,y}\in\mathcal{S}(n,2j,G)$.
\end{Definition}
 Note that when $n=1$, this process reduces to the single-species SEP(2j) defined in \cite{Giardin__2009}.

Suppose $p=(p_0,\dots,p_n)$ is a probability distribution, the multinomial distribution on $\Omega_{2j}$ is defined as
\begin{gather*}
 w_p(\xi)=\binom{2j}{\xi}\prod_{i=0}^np_i^{\xi_i},
\end{gather*}
where $ \binom{2j}{\xi}$ denotes the multinomial coefficient $\frac{(2j)!}{\prod_{i=0}^n\xi_i!}$. Following a simple detailed balance computation, we can show that the product measure with marginals $w_p(\xi)$ for any fixed $p$ being a distribution is a reversible measure of the multi-species SEP$(2j)$, i.e., $\otimes_G w_p$ is a reversible measure of the multi-species SEP$(2j)$ when $p$ is the same for all sites.

\subsection[Multivariate Krawtchouk polynomials and Lie algebra sl\_\{n+1\}]{Multivariate Krawtchouk polynomials and Lie algebra $\boldsymbol{\mathfrak{sl}_{n+1}}$}\label{section3.1}

First, we introduce the $n$-variable Krawtchouk polynomials defined by Griffiths \cite{https://doi.org/10.1111/j.1467-842X.1971.tb01239.x}. We shall adopt the notation of Iliev~\cite{iliev_2012} in the following.

\begin{Definition}\label{def:1}
Let $\mathcal{K}_n$ be the set of 4-tuples $\big(\nu,P,\hat{P},U\big)$ such that
 $P$, $\hat{P}$, $U$ are $(n+1)\times(n+1)$ complex matrices satisfying the following conditions:
\begin{enumerate}\itemsep=0pt
\item[(1)] $P=\operatorname{diag}(p_0,\dots,p_n)$, $\hat{P}=\operatorname{diag}(\hat{p}_0,\dots,\hat{p}_n)$ and $p_0=\hat{p}_0=\frac{1}{\nu}\neq 0$,
\item[(2)] $U=(u_{kl})_{0\le k,l\le n}$ with $U_{k0}=U_{0k}=1$ for all $0\le k \le n$,
\item[(3)] $\nu PU\hat{P}U^{\rm T}=I_{n+1}$.
\end{enumerate}
\end{Definition}

It follows from the above definition that $p=(p_0,\dots,p_n)$ and $\hat{p}=(\hat{p}_0,\dots,\hat{p}_n)$ satisfy that $\sum_{k=0}^n p_k=\sum_{k=0}^n \hat{p}_k=1$ and $p_k,\hat{p}_k\neq 0$ for any $k$.

For all points $\kappa\in \mathcal{K}_n$, Griffith constructed multivariate Krawtchouk polynomials using a~generating function as follows.
 \begin{Definition}[\cite{https://doi.org/10.1111/j.1467-842X.1971.tb01239.x}]
 For $\xi,\eta\in \Omega_{2j}$ and $\kappa\in \mathcal{K}_n$, the multivariate Krawtchouk polynomial $K(\xi,\eta,\kappa,j)$ is defined by
 \begin{gather*}
 \sum_{\xi\in\Omega_{2j}}\binom{2j}{\xi}K(\xi,\eta,\kappa,j)z_1^{\xi_1}\cdots z_n^{\xi_n}=\prod_{k=0}^n\bigg(1+\sum_{l=1}^nu_{kl}z_l\bigg)^{\eta_k}.
 \end{gather*}
 \end{Definition}

Although it's not obvious to tell from the generating function, $K(\xi,\eta,\kappa,j)$ depends on~$P$ and~$\hat{P}$ because the matrix $U\in\kappa$ satisfies the condition~(3) in Definition~\ref{def:1}.

In what follows, we fix a 4-tuple $\kappa\in \mathcal{K}_n$ as in Definition~\ref{def:1}, we also write $K(\xi,\eta,\kappa,j) $ as~$K(\xi,\eta)$ for simplicity.

In \cite{iliev_2012}, Iliev interpreted multivariate Krawtchouk polynomials with representations of the Lie algebra~$\mathfrak{sl}_{n+1}$. We recall some of the essential results. Let's start by introducing some basic notations.
Let $z_0,\dots,z_n$ be mutually commuting variables, we set $z=(z_0,\dots,z_n)$. For each $\xi=(\xi_0,\dots,\xi_n)\in\Omega_{2j}$, we denote
$z^\xi=z_0^{\xi_0}z_1^{\xi_1}\cdots z_n^{\xi_n}$ and $\xi!=\xi_0!\cdots\xi_n!$. Also define $V_{2j}=\operatorname{span}\big\{z^\xi|\xi\in\Omega_{2j}\big\}\subset \mathbb{C}[z]$, which is the space consisting of all homogeneous complex polynomials of total degree $2j$.

Let $I_{n+1}$ denote the $(n+1)\times(n+1)$ identity matrix. For $0\le k,l\le n$, let $e_{k,l}$ denote the $(n+1)\times(n+1)$ matrix with $(k,l)^{\rm th}$ entry $1$ and other entries~$0$.
The special linear Lie algebra of order $n+1$ denoted by $\mathfrak{sl}_{n+1}$ consists of $(n+1)\times (n+1) $ matrices with trace zero and has the Lie bracket $[X,Y]=XY-YX$. It has basis $\{e_{kl}\}_{0\le k\neq l\le n}$ and $\{h_l\}_{0< l\le n}$, where $h_l=e_{ll}-\frac{1}{n+1}I_{n+1}$.

The $*$-structure of $\mathfrak{sl}_{n+1}$ is given by
\begin{gather}\label{eq:1}
 e_{kl}^*=e_{lk}, \quad k\neq l , \qquad h_l^*=h_l.
\end{gather}

Let ${\mathfrak {gl}}(V_{2j})$ denote the space of endomorphisms of ${V_{2j}}$, we consider the representation $\rho\colon \mathfrak{sl}_{n+1}\allowbreak \xrightarrow{} \mathfrak{gl}(V_{2j})$ defined by
\begin{gather}\label{eq:3.2}
 \rho e_{kl} =z_k\partial z_l,\quad k\neq l,\qquad \rho h_l=z_l\partial z_l-\frac{2j}{n+1}.
\end{gather}

Next, define an antiautomorphism $\mathfrak{a}$ on $\mathfrak{sl}_{n+1}$ by $\mathfrak{a}(X)=\hat{P}X^{\rm T}\hat{P}^{-1}$.
It follows easily that
\begin{gather}\label{eq:3.1}
 \mathfrak{a}(e_{kl})=\frac{\hat{p}_l}{\hat{p}_k}e_{lk},\quad k\neq l, \qquad
 \mathfrak{a}(h_{l})=h_l.
\end{gather}

We define a symmetric bilinear form $\langle\,,\,\rangle_{\kappa}$ on $V_{2j}$ by
\begin{gather*}
\big\langle z^\xi,z^\eta\big\rangle _{\kappa}=\delta_{\xi,\eta}\frac{\xi!}{\hat{p}^\xi}\theta^{2j}.
\end{gather*}

Then it is easy to check that for any $X\in \mathfrak{sl}_{n+1}$ and $v_1,v_2\in V_{2j}$
\begin{gather}\label{eq:2}
\langle \rho Xv_1,v_2\rangle _{\kappa}=\langle v_1,\rho \mathfrak{a}(X)v_2\rangle_{\kappa}.
\end{gather}

Let $R$ be the matrix
\begin{gather}\label{eq:4}
 R=\hat{\theta}\hat{P}U^{\rm T},
\end{gather}
where $\hat{\theta}\in \mathbb{C}$ such that $\det(R)=1$.
Next, we define $\hat{z}=(\hat{z}_0,\dots,\hat{z}_n)$ by $\hat{z}=zR$.
\begin{Lemma}\label{lemma:3.1}
Define operator $\operatorname{Ad}_R$ on $\mathfrak{sl}_{n+1}$ by $\operatorname{Ad}_R(X)=R^{-1}XR$, where $R=(r_{kl})_{0\le k,l\le n}$ is defined in equation~\eqref{eq:4}. Then $\operatorname{Ad}_R$ is a Lie algebra automorphism of $\mathfrak{sl}_{n+1}$.
\end{Lemma}
\begin{proof}
 It can be checked directly.
\end{proof}

Last, we list some properties of the multivariate Krawtchouk polynomial whose proof could be found in~\cite{iliev_2012}.
\begin{Proposition}[{\cite[Corollary 5.2]{iliev_2012}}]\label{prop:3.1}
For $\xi,\eta\in \Omega_{2j}$, the multivariate polynomial~$K$ has the following bilinear form,
\begin{gather*}
 K(\xi,\eta)=\frac{p_0^{2j}}{(2j)!}\big\langle z^\xi,\hat{z}^\eta\big\rangle_{\kappa}.
\end{gather*}
\end{Proposition}

\begin{Proposition}[{\cite[Corollary 5.3]{iliev_2012}}]\label{prop:3.2}
For $\xi,\eta,\zeta\in \Omega_{2j}$ we have the following relations:
\begin{gather*}
 \sum_{\xi\in\Omega_{2j}}K(\xi,\eta)K(\xi,\zeta)w_{\hat{p}}(\xi)=\frac{p_0^{2j}\delta_{\eta,\zeta}}{w_p(\eta)},
\\
 \sum_{\xi\in\Omega_{2j}}K(\eta,\xi)K(\zeta,\xi)w_p(\xi)=\frac{p_0^{2j}\delta_{\eta,\zeta}}{w_{\hat{p}}(\eta)}.
\end{gather*}
\end{Proposition}
\begin{Remark}
If $U\in\kappa$ is a real matrix, then it follows from the generating function that the multivariate Krawtchouk polynomial is real valued. In this case, Proposition~\ref{prop:3.2} is the orthogonality relation for the multivariate Krawtchouk polynomial in $l^2(w_p)$ and $l^2(w_{\hat{p}})$.
\end{Remark}

\subsection[Self-duality of the multi-species SEP(2j)]{Self-duality of the multi-species SEP$\boldsymbol{(2j)}$}\label{section3.2}

In this subsection, we show that the multi-species SEP($2j$) is self dual with respect to duality functions given by homogeneous products of multivariate Krawtchouk polynomials. Suppose $p$ and $\hat{p}$ in the 4-tuple $\kappa\in \mathcal{K}_n$ as in Definition~\ref{def:1} are both probability measures.

Let $l^2(w_p)$ be a Hilbert space with inner product $(f,g)_p=\sum_{\xi\in\Omega_{2j}}f(\xi)\overline{g(\xi)}w_p(\xi)$. Now we define a $*$-representations $\rho_p$ of $\mathfrak{sl}_{n+1}$ on $l^2(w_p)$ by
\begin{alignat*}{3}
 &\rho_p(e_{kl})f(\xi)=\sqrt{\frac{p_k}{p_l}}\xi_lf\big(\xi_{l,k}^{-1,+1}\big) \qquad &&\text{for} \ 0\le k\neq l\le n,& \\
 &\rho_p(h_l)f(\xi)=\left(\xi_l-\frac{2j}{n+1}\right)f(\xi)\qquad&& \text{for} \ 0< l\le n,&
\end{alignat*}
where $\xi_{l,k}^{+1,-1}$ represents the variable with $\xi_l$ increased by 1 and $\xi_k$ decreased by 1. Recalling the $*$-structure defined in equation~\eqref{eq:1}, it is straightforward to check that $(\rho_p(X)f,g)_p=(f,\rho_p(X^*)g)_p$ for all $X\in \mathfrak{sl}_{n+1} $.

Next, we introduce another non-trivial $*$-representation $\sigma_p$ of $\mathfrak{sl}_{n+1}$ on $l^2(w_{p})$ that is unitarily equivalent to $\rho_{\hat{p}}$.
\begin{Definition}\label{def:2}
For each $\rho_{\hat{p}}$, define a corresponding representation by $\sigma_p=\rho_{\hat{p}}\circ \operatorname{Ad}_R $, where~$\operatorname{Ad}_R$ is the automorphism defined in Lemma~\ref{lemma:3.1}.
\end{Definition}

\begin{Proposition}\label{prop:3.4}
If the matrix $U$ in the $4$-tuple $\kappa\in \mathcal{K}_n$ is real, then the representation $\sigma_p$ defined in Definition~{\rm \ref{def:2}} is a $*$-representation of~$\mathfrak{sl}_{n+1}$ on~$l^2(w_{p})$.
\end{Proposition}

\begin{proof}
By definitions of the matrices $U$ and $R$, when $U$ is a real matrix, then $R$ and $R^{-1}$ are all real matrices. For ease of notation, we write $Q=R^{-1}=(q_{i,m})_{0\le i,m\le n}$.
Then, computing $\sigma_p$ explicitly, we have
\begin{gather}\label{eq:5.3}
 \sigma_p(e_{im})f(\eta)=\sqrt{\frac{\hat{p}_i}{\hat{p}_m}}\sum_{k,l=0}^n q_{ki}r_{ml}\eta_lf\big(\eta_{k,l}^{+1,-1}\big).
\end{gather}

Next we verify that for any $X\in \mathfrak{sl}_{n+1} $, $\left(\sigma_p(X) f(\eta),g(\eta)\right)_{p}= ( f(\eta),\sigma_p(X^*)g(\eta) )_{p}$.
First, we plug equation~\eqref{eq:5.3} in the inner products, when $i\neq m$,
\begin{gather}
 \big(\sigma_p(e_{im}) f(\eta),g(\eta)\big)_{p}=\sqrt{\frac{\hat{p}_i}{\hat{p}_m}}\sum_{k,l=0}^n q_{ki}r_{ml}\big(\eta_lf\big(\eta_{k,l}^{+1,-1}\big),g(\eta)\big)_{p},\nonumber\\
\label{eq:5.1}
\big( f(\eta),\sigma_p(e_{mi})g(\eta)\big)_{p}=\sqrt{\frac{\hat{p}_m}{\hat{p}_i}}\sum_{k,l=0}^n q_{km}r_{il}\big(f(\eta),\eta_lg\big(\eta_{k,l}^{+1,-1}\big)\big)_{p}.
\end{gather}
By switching $k$ and $l$ in equation~\eqref{eq:5.1}, we have that
\begin{gather*}
 \big( f(\eta),\sigma_p(e_{mi})g(\eta)\big)_{p}=\sqrt{\frac{\hat{p}_m}{\hat{p}_i}}\sum_{k,l=0}^n q_{lm}r_{ik}\big(\eta_kf(\eta),g\big(\eta_{l,k}^{+1,-1}\big)\big)_{p}.
\end{gather*}
Now define $\tilde{\eta}=\eta_{l,k}^{+1,-1}$, we get \begin{gather*}
\big( f(\eta),\sigma_p(e_{mi})g(\eta)\big)_{p}=\sqrt{\frac{\hat{p}_m}{\hat{p}_i}}\sum_{k,l=0}^n q_{lm}r_{ik}\frac{p_k}{p_l}\big(\tilde{\eta}_lf\big(\tilde{\eta}_{l,k}^{-1,+1}\big),g(\tilde{\eta})\big)_{p}.
\end{gather*}

Computing the entries of matrices $R$ and $Q$ explicitly, we have $r_{ki}=\hat{\theta}\hat{p}_{k}u_{ik}$ and $q_{ki}=\frac{1}{p_0\hat{\theta}}p_ku_{ki}$. Plugging in $r$ and $q$, we have that $\frac{\hat{p}_i}{\hat{p}_m}q_{ki}r_{ml}=\frac{p_k}{p_l}q_{lm}r_{ik}$, which gives that
\[
\big(\sigma_p(e_{im}) f(\eta),g(\eta)\big)_{p}=\big( f(\eta),\sigma_p(e_{mi})g(\eta)\big)_{p}.
\]

The proof for $h_l$ is similar.
\end{proof}
\begin{Remark}
Proposition~\ref{prop:3.4} is nontrivial since the automorphism $\operatorname{Ad}_R$ does not preserve the $*$-structure of $\mathfrak{sl}_{n+1}$, i.e., there exists $X\in \mathfrak{sl}_{n+1}$ such that $\operatorname{Ad}_R(X^*)\neq \operatorname{Ad}_R(X)^*$. For example,
\[
\operatorname{Ad}_R(e_{12}^*)=\operatorname{Ad}_R(e_{21})=\sum_{k,l}q_{k2}r_{1l}e_{kl},
\] while
\[
\operatorname{Ad}_R(e_{12})^*=\bigg(\sum_{k,l}q_{k1}r_{2l}e_{kl}\bigg)^*=\sum_{k,l}q_{k1}r_{2l}e_{lk}=\sum_{k,l}q_{l1}r_{2k}e_{kl},
\] which are not equal for general~$\kappa$.
\end{Remark}

\begin{Proposition}
When the matrix $U$ in the $4$-tuple $\kappa$ is real, $\rho_{\hat{p}}$ and $\sigma_p$ satisfy the following property: $[\rho_{\hat{p}}(X^*)K(\cdot,\eta)](\xi)=[\sigma_p(X)K(\xi,\cdot)](\eta)$ for any $X\in\mathfrak{sl}_{n+1}$ and $(\xi,\eta)\in \Omega_{2j}\times \Omega_{2j}$.
\end{Proposition}

\begin{proof}
Recall from Proposition~\ref{prop:3.1}, the multivariate Krawtchouk polynomials can be written in bilinear form. Also recall that the antiautomorphism $\mathfrak{a}$ defined in \eqref{eq:3.1} has property~\eqref{eq:2}. Notice that for function~$z^\xi$, we can write~$\rho_{p}$ in terms of the representation $\rho$ defined in~\eqref{eq:3.2},
\begin{gather*}
 \rho_{\hat{p}}(e_{im})z^{\xi}=\sqrt{\frac{\hat{p}_i}{\hat{p}_m}}\rho(e_{im})z^{\xi}.
\end{gather*}
Thus, for $i\neq m$,
\begin{align*}
 [\rho_{\hat{p}}(e_{im})K(\cdot,\eta)](\xi)& =\frac{p_0^{2j}}{(2j)!}\sqrt{\frac{\hat{p}_i}{\hat{p}_m}}\big\langle \rho(e_{im})z^{\xi},\hat{z}^\eta\big\rangle_{\kappa} =\frac{p_0^{2j}}{(2j)!}\sqrt{\frac{\hat{p}_i}{\hat{p}_m}}\big\langle z^{\xi},\rho\mathfrak{a}(e_{im})\hat{z}^\eta\big\rangle _{\kappa}\\
& =\frac{p_0^{2j}}{(2j)!}\sqrt{\frac{\hat{p}_m}{\hat{p}_i}}\big\langle z^{\xi},\rho(e_{mi})\hat{z}^\eta\big\rangle _{\kappa}=[\sigma_p(e_{mi})K(\xi,\cdot)](\eta),
\end{align*}
where the last equality follows form the fact that $e_{im}=\operatorname{Ad}_R(\hat{e}_{im})$ and $\rho(\hat{e}_{im})\hat{z}^\eta=\hat{z}_i\partial_{\hat{z}_m}\hat{z}^\eta$~\cite{iliev_2012}, thus \[ \sqrt{\frac{\hat{p}_m}{\hat{p}_i}}\rho(e_{mi})\hat{z}^\eta=\sqrt{\frac{\hat{p}_m}{\hat{p}_i}}\rho\circ \operatorname{Ad}_R(\hat{e}_{mi})\hat{z}^\eta=\sigma_p(e_{mi})\hat{z}^\eta.\]

The proof for $h_l$ follows from the same argument.
\end{proof}

\begin{Proposition}
 If the matrix $U$ in the $4$-tuple $\kappa$ is real, define the operator $\Lambda\colon l^2(w_{\hat{p}})\xrightarrow[]{} l^2(w_p)$ by
\begin{gather*}
 (\Lambda f)(\eta)=p_0^{-j}\sum_{\xi\in \Omega_{2j}}w_{\hat{p}}(\xi)f(\xi)K(\xi,\eta).
\end{gather*}

Then $\Lambda$ is an unitary operator and intertwines $\rho_{\hat{p}}$ with $\sigma_p$.
The kernel $K(\xi,\eta)$ satisfies
\begin{gather}\label{eq:3}
 [\rho_{\hat{p}}(X^*)K(\cdot,\eta)](\xi)=[\sigma_p(X)K(\xi,\cdot)](\eta).
\end{gather}

\end{Proposition}

\begin{proof}It follows directly from equation~\eqref{eq:3} that $\Lambda[\rho_{\hat{p}}(X)f]=\sigma_p(X)\Lambda(f)$ for all $X\in\mathfrak{sl}_{n+1}$, thus $\Lambda$ intertwines $\rho_{\hat{p}}$ with $\sigma_p$. Recall that if a process has reversible measure, then it's self dual with respect to the cheap duality function, which comes from the reversible measure. For the multi-species SEP$(2j)$, the cheap duality function is given by $\delta_{\zeta}(\xi)=\frac{\delta_{\zeta,\xi}}{w_{\hat{p}}(\xi)}$, which has squared norm $\frac{1}{w_{\hat{p}}(\zeta)}$ in $l^2(w_{\hat{p}})$.

On the other hand, $\Lambda(\delta_\zeta)(\eta)=p_0^{-j}K(\zeta,\eta)$ has squared norm $\frac{1}{w_{\hat{p}}(\zeta)}$ in $l^2(w_p)$. Thus $\Lambda$ maps an orthogonal basis to another orthogonal basis preserving the norm, hence $\Lambda$ is unitary.
\end{proof}

Last we show the generator defined in~\eqref{eq:6} is the image of some self-adjoint element in $\mathcal{U}(\mathfrak{sl}_{n+1})^{\otimes L}$ under the $*$-representations $\rho_{\hat{p}}^{\otimes L}$ and $\sigma_p^{\otimes L}$. We generalize the construction of the Markov generator in terms of the co-product of a Casimir element (see, e.g., \cite{Giardin__2009,Groenevelt_2018}) to the multi-species cases.

We start by constructing a Casimir element of $\mathcal{U}(\mathfrak{sl}_{n+1})$. Under the non-degenerate bilinear form $B(X,Y)=\operatorname{tr}(XY)$, the dual basis of $\mathfrak{sl}_{n+1}$ is given by
\begin{gather*}
 e_{lk}^\star=e_{kl}, \quad k\neq l, \qquad h_l^\star=e_{ll}-e_{00}=h_l+\sum_{k=1}^nh_k.
\end{gather*}
The Casimir element $\Omega$ of $\mathcal{U}(\mathfrak{sl}_{n+1})$ is given by
\begin{gather*}
 \Omega=\sum_{0\le k<l\le n}(e_{kl}e_{lk}+e_{lk}e_{kl})+\sum_{0< l\le n} h_lh_l^\star .
\end{gather*}
It is easy to verify that $\Omega$ is self-adjoint, i.e., with the $*$-structure given in~\eqref{eq:1}, $\Omega^*=\Omega$.
Next, define the coproduct for the basis $\{e_{kl}\}_{0\le k\neq l\le n}$ and $\{h_l\}_{0< l\le n}$ as $\Delta(X)=1\otimes X+X\otimes 1$, and define an element
\begin{gather*}
 Y=\Delta(\Omega)-\Omega\otimes 1-1\otimes \Omega.
\end{gather*}

\begin{Lemma}
$\mathcal{L}_{x,y}$ is the image of a self-adjoint element in $\mathcal{U}(\mathfrak{sl}_{n+1})^{\otimes 2}$ under the representation $\rho_{\hat{p}}\otimes\rho_{\hat{p}}$ and $\sigma_p\otimes\sigma_p$. Specifically, there exists a constant $c\in \mathbb{R}$ such that
\begin{align}\label{eq:5}
 \mathcal{L}_{x,y}& =\frac{1}{2}\rho_{\hat{p}}\otimes\rho_{\hat{p}} (Y_{x,y} )-c
\\ \label{eq:5.2}
& = \frac{1}{2}\sigma_p\otimes\sigma_p (Y_{x,y} )-c.
\end{align}
\end{Lemma}

\begin{proof}To prove \eqref{eq:5}, we make use of the following identity:
\begin{gather*}
 \sum_{0\le k<l\le n}\xi_k^x\xi_l^y+\xi_k^x\xi_l^y=\bigg(\sum_{0\le l\le n}\xi_l^x\bigg)\bigg(\sum_{0\le l\le n}\xi_l^y\bigg)-\sum_{0\le l \le n} \xi^x_l\xi^y_l=(2j)^2-\sum_{0\le l \le n} \xi^x_l\xi^y_l.
\end{gather*}
Expanding $\Omega$ in $Y$, we have
\begin{gather*}
 Y=2\sum_{0\le k< l\le n} ( e_{lk}\otimes e_{kl}+e_{kl}\otimes e_{lk} )+\sum_{1\le l\le n} (h_l\otimes h_l^{\star}+h_l^{\star}\otimes h_l ).
\end{gather*}
Now we can compute the right-hand side of~\eqref{eq:5} using the above identities and the representation $\rho_{\hat{p}}$ to see that it agrees with $\mathcal{L}_{x,y}$. Note that $\mathcal{L}_{x,y}$ does not depend on~$\hat{p}$ since all terms with~$\hat{p}$ get cancelled.

To prove \eqref{eq:5.2}, it suffices to show $\operatorname{Ad}_R\otimes \operatorname{Ad}_R(Y)=Y$. First, we can check that $\operatorname{Ad}_R(\Omega)=\Omega$ by direct calculation.
Using the fact that $\operatorname{Ad}_R\otimes \operatorname{Ad}_R \circ \Delta=\Delta\circ \operatorname{Ad}_R$, we have $\operatorname{Ad}_R\otimes \operatorname{Ad}_R(Y)\allowbreak =Y$, thus
\begin{gather*}
\sigma_p \otimes\sigma_p(Y_{x,y})= (\rho_{\hat{p}}\otimes\rho_{\hat{p}} )\circ (\operatorname{Ad}_R\otimes \operatorname{Ad}_R ) (Y_{x,y}) =\rho_{\hat{p}}\otimes\rho_{\hat{p}} (Y_{x,y}).\tag*{\qed}
\end{gather*}
 \renewcommand{\qed}{}
\end{proof}

Applying Theorem~\ref{th:2.1} yields the self duality for the multi-species SEP$(2j)$.
\begin{Theorem}
 The multi-species ${\rm SEP}(2j)$ defined in Definition~{\rm \ref{def:3.1}} is self dual with respect to duality functions
\begin{gather*}
 \prod_{x\in G} K\big(\xi^x,\eta^x,\kappa,2j\big),
\end{gather*}
 for any $\kappa\in \mathcal{K}_n$ such that $U$ in $\kappa$ is real and~$p$, $\hat{p}$ in $\kappa$ are probability measures.
\end{Theorem}

\section[Multi-species IRW and Heisenberg Lie algebra h\_n]{Multi-species IRW and Heisenberg Lie algebra $\boldsymbol{\mathfrak{h}_n}$}\label{section4}

In this section, we find a family of self-duality functions of the multi-species independent random walk (multi-species IRW) using the Heisenberg Lie algebra $\mathfrak{h}_n$.

The $n$-species independent random walk is a generalization of the usual IRW to~$n$ species on a finite undirected graph $G=(V,E)$, where $V=\{1,\dots,L\}$ with $L\in\mathbb{N}$ and $L>2$ is the set of sites (vertices) and $E$ is the set of edges. It is a Markov process where $n$ species of particles move independently between $L$ sites. The jump rate for a particle of species $i$ from a site is proportional to the number of species $i$ particles at that site.

The state space $\mathcal{S}(n,G)$ of particle configurations consists of variables $\xi=\big(\xi_i^x\colon 1\le i\le n$, $x\in G\big)$, where $\xi_i^x\in\mathbb{N}_0$ (non-negative integers) denotes the number of species~$i$ particles at site~$x$.
\begin{Definition}\label{def:4.1}
The generator of $n$-species IRW on $G=(V,E)$ is given by
\begin{gather*}
 \mathcal{L}f(\xi)=\sum_{\text{edge}\{x,y\}\in E} \mathcal{L}_{x,y},\\
 \mathcal{L}_{x,y}= \sum_{i=1}^n\big [\xi_i^x\big(f\big(\xi_i^{x,y}\big)-f(\xi)\big)+\xi_i^y\big(f\big(\xi_i^{y,x}\big)-f(\xi)\big)\big],
\end{gather*}
where $\xi_i^{x,y}$ denotes the particle configuration obtained by moving a particle of species~$i$ from site~$x$ to site~$y$ if $\xi_i^{x,y}\in\mathcal{S}(n,G)$.
\end{Definition}

Define measure $ \mu_\lambda(\xi)=\prod_{i=1}^n \frac{\lambda^{\xi_i}}{\xi_i!}{\rm e}^{-\lambda}$ with $ \lambda>0$. Following a simple detailed balance computation, we can show that the product measure $\otimes_G\mu_\lambda$ is a reversible measure of the $n$-species IRW when $\lambda$ is the same for all sites.

Next, we mention here that the space $l^2(\mu_\lambda)$ is equipped with inner product
\begin{gather*}
 (f,g)_\lambda=\sum_{\xi\in\mathbb{N}_0^{ n}}\mu_\lambda(\xi)f(\xi)\overline{g(\xi)}.
\end{gather*}

\subsection[The Charlier polynomials and Heisenberg Lie algebra h\_n]{The Charlier polynomials and Heisenberg Lie algebra $\boldsymbol{\mathfrak{h}_n}$}\label{section4.1}

\begin{Definition}
The Heisenberg Lie algebra $\mathfrak{h}_n$ is the $2n+1$ dimensional complex Lie algebra with generators $\{P_1,\dots,P_n,Q_1,\dots,Q_n,Z\}$ and commutation relations: for $1\le i,l \le n$,
\begin{gather*}
 [P_i,P_l]=[Q_i,Q_l]=[P_i,Z]=[Q_i,Z]=0,\qquad
 [P_i,Q_l]=\delta_{i,l}Z.
\end{gather*}
\end{Definition}
The Heisenberg Lie algebra $\mathfrak{h}_n$ is nilpotent but not semisimple. It has a $*$-structure given by
\begin{gather*}
 P_i^*=Q_i,\qquad Q_i^*=P_i,\qquad Z^*=Z.
\end{gather*}

The Charlier polynomials are given by
\begin{gather*}
 C_m(z,\lambda)={}_2F_0\left(\left.\begin{matrix}
 -m,\ -z \\
 -
\end{matrix}\,\right| -\frac{1}{\lambda}\right).
\end{gather*}
Here we list some properties of Charlier polynomials that will be used later on.
First, Charlier polynomials are Orthogonal,
 \begin{gather*}
 \sum_{z\in \mathbb{N}_0}C_m(z,\lambda)C_{\tilde{m}}(z,\lambda)\frac{\lambda^{z}}{z!}{\rm e}^{-\lambda}=\delta_{m,\tilde{m}}\lambda^{-\tilde{m}}\tilde{m}!.
 \end{gather*}
 They have the following raising and lowering property,
 \begin{gather*}
 mC_{m-1}(z,\lambda)=\lambda C_m(z,\lambda)-\lambda C_m(z+1,\lambda),
 \\
 \lambda C_{m+1}(z,\lambda)=\lambda C_m(z,\lambda)-z C_m(z-1,\lambda).
 \end{gather*}

To construct unitary operator later, we define function $C(\xi,\eta,\lambda)$ for $\xi,\eta\in \mathbb{N}_0^{ n}$ by
\begin{gather}\label{eq: 4.9}
 C(\xi,\eta,\lambda)=\prod_{i=1}^n {\rm e}^{\lambda} C_{\xi_i}(\eta_i,\lambda).
\end{gather}

\subsection{Self duality of the multi-species IRW}\label{section4.2}
Now we define the $*$-representation $\rho_\lambda$ of $\mathfrak{h}_n$ on $l^2(\mu_\lambda)$ by
\begin{gather*}
 [\rho_\lambda(Q_i)f](\xi)=\xi_if\big(\xi_i^{-1}\big),\\
 [\rho_\lambda(P_i)f](\xi)=\lambda f\big(\xi_i^{+1}\big),\\
 [\rho_\lambda(Z)f](\xi)=\lambda f(\xi),
\end{gather*}
where $\xi_i^{+1}$ \big($\xi_i^{-1}$\big) means that $\xi_i$ is increased (decreased) by~$1$.

Next, we define the map $\theta$ by
\begin{gather*}
 \theta(P_i)=Z-P_i,\qquad \theta(Q_i)=Z-Q_i,\qquad \theta(Z)=Z,
\end{gather*}
then $\theta$ extends to a Lie algebra isomorphism of $\mathfrak{h}_n$, preserving the $*$-structure.

\begin{Proposition}\label{prop:4.1}
 For any $X\in \mathfrak{h}_n$, we have
 \begin{gather*}
 \rho_\lambda(X^*)C(\cdot,\eta,\lambda)(\xi)=\rho_\lambda(\theta(X))C(\xi,\cdot,\lambda)(\eta).
 \end{gather*}
\end{Proposition}

\begin{proof}
It's easy to verify by definition and the raising and lowering property,
\begin{gather*}
 \rho_\lambda(Q_i)C(\cdot,\eta,\lambda)(\xi)=\xi_i C\big(\xi_i^{-1},\eta,\lambda\big)\\
 \hphantom{\rho_\lambda(Q_i)C(\cdot,\eta,\lambda)(\xi)}{} =\lambda C(\xi,\eta,\lambda)-\lambda C\big(\xi,\eta_i^{+1},\lambda\big)
 =\rho_\lambda(\theta(P_i))C(\xi,\cdot,\lambda)(\eta),\\
 \rho_\lambda(P_i)C(\cdot,\eta,\lambda)(\xi)=\lambda C\big(\xi_i^{+1},\eta,\lambda\big)\\
 \hphantom{\rho_\lambda(P_i)C(\cdot,\eta,\lambda)(\xi)}{}
 =\lambda C(\xi,\eta,\lambda)-\eta_i C\big(\xi,\eta_i^{-1},\lambda\big)
 =\rho_\lambda(\theta(Q_i))C(\xi,\cdot,\lambda)(\eta) .\tag*{\qed}
\end{gather*}
\renewcommand{\qed}{}
\end{proof}

\begin{Proposition}
 Define the operator $\Lambda\colon l^2(\mu_\lambda)\xrightarrow[]{} l^2(\mu_\lambda)$ by
\begin{gather*}
 (\Lambda f)(\eta)=\sum_{\xi\in \mathbb{N}_0^{ n}}\mu_\lambda(\xi)f(\xi)C(\xi,\eta,\lambda),
\end{gather*}
then $\Lambda$ is an unitary operator and intertwines $\rho_\lambda$ with $\rho_\lambda\circ\theta$.
\end{Proposition}
\begin{proof}
It follows directly from Proposition~\ref{prop:4.1} that $\Lambda[\rho_\lambda(X)f]=\rho_\lambda\circ\theta(X)\Lambda(f)$ for all $X\in\mathfrak{h}_n$, thus $\Lambda$ intertwines $\rho_{\lambda}$ with $\rho_\lambda\circ\theta$. The cheap duality functions for the $n$-species IRW given by $\delta_{\zeta}(\xi)=\frac{\delta_{\zeta,\xi}}{\mu_\lambda(\xi)}$ form an orthogonal basis for $l^2(\mu_\lambda)$ with squared norm $\frac{1}{\mu_\lambda(\zeta)}$, while $\Lambda(\delta_\zeta)(\eta)=C(\zeta,\eta,\lambda)$ also has squared norm $\frac{1}{\mu_\lambda(\zeta)}$ in $l^2(\mu_\lambda)$. By the fact that all $C(\zeta,\eta,\lambda)$ form an orthogonal basis for $l^2(\mu_\lambda)$, $\Lambda$ is unitary.
\end{proof}

Finally, to show self duality, we define $Y\in \mathcal{U}(\mathfrak{h}_n)^{\otimes 2}$ by
\begin{gather*}
 Y=\sum_{i=1}^n (1\otimes Q_i-Q_i\otimes 1)(P_i\otimes 1-1\otimes P_i).
\end{gather*}

\begin{Lemma} The generator for the multi-species IRW can be written as the following:
\begin{align}\label{eq:7}
 \mathcal{L}_{x,y}& =\lambda^{-1} \rho_\lambda\otimes\rho_\lambda (Y_{x,y})
\\ \label{eq:8}
 & =\lambda^{-1} (\rho_\lambda\circ\theta)\otimes(\rho_\lambda \circ\theta)(Y_{x,y}).
\end{align}
\end{Lemma}
\begin{proof}
 \eqref{eq:7} is obtained by plugging in definitions, and to prove \eqref{eq:8}, we show that $(\rho_\lambda\circ\theta)\otimes(\rho_\lambda \circ\theta)(Y)=\rho_\lambda\otimes\rho_\lambda (Y)$, which follows from \cite[Lemma~3.5]{Groenevelt_2018}.
\end{proof}

Again, applying Theorem~\ref{th:2.1}, we obtain the self duality for the multi-species IRW.
\begin{Theorem}
The multi-species IRW defined in Definition~{\rm \ref{def:4.1}} is self dual with respect to duality functions
\begin{gather}\label{eq: 4.20}
 \prod_{x\in G} C\big(\xi^x,\eta^x,\lambda\big),\qquad \lambda>0.
\end{gather}
\end{Theorem}
\begin{Remark}
These duality functions could be obtained by the independence of the evolution of each species of particles and the fact that the duality functions given by \eqref{eq: 4.9} and~\eqref{eq: 4.20} suitably factorize over
species.
\end{Remark}

\subsection*{Acknowledgments}
The author is very grateful to Jeffrey Kuan and anonymous referees for helpful discussions and insightful comments.

\pdfbookmark[1]{References}{ref}
\LastPageEnding


\begin{thebibliography}{99}
\footnotesize\itemsep=0pt

\bibitem{Ayala_2018}
Ayala M., Carinci G., Redig F., Quantitative {B}oltzmann--{G}ibbs principles
 via orthogonal polynomial duality, \href{https://doi.org/10.1007/s10955-018-2060-7}{\textit{J.~Stat. Phys.}} \textbf{171}
 (2018), 980--999, \href{https://arxiv.org/abs/1712.08492}{arXiv:1712.08492}.

\bibitem{ayala2020higher}
Ayala M., Carinci G., Redig F., Higher order fluctuation fields and orthogonal
 duality polynomials, \href{https://doi.org/10.1214/21-EJP586}{\textit{Electron.~J. Probab.}} \textbf{26} (2021), 27,
 35~pages, \href{https://arxiv.org/abs/2004.08412}{arXiv:2004.08412}.

\bibitem{Borodin_2014}
Borodin A., Corwin I., Sasamoto T., From duality to determinants for
 {$q$}-{TASEP} and {ASEP}, \href{https://doi.org/10.1214/13-AOP868}{\textit{Ann. Probab.}} \textbf{42} (2014),
 2314--2382, \href{https://arxiv.org/abs/1207.5035}{arXiv:1207.5035}.

\bibitem{Caputo2008ONTS}
Caputo P., On the spectral gap of the {K}ac walk and other binary collision
 processes, \textit{ALEA Lat. Am.~J. Probab. Math. Stat.} \textbf{4} (2008),
 205--222, \href{https://arxiv.org/abs/0807.3415}{arXiv:0807.3415}.

\bibitem{Carinci_2019}
Carinci G., Franceschini C., Giardin\`a C., Groenevelt W., Redig F., Orthogonal
 dualities of {M}arkov processes and unitary symmetries, \href{https://doi.org/10.3842/SIGMA.2019.053}{\textit{SIGMA}}
 \textbf{15} (2019), 053, 27~pages, \href{https://arxiv.org/abs/1812.08553}{arXiv:1812.08553}.

\bibitem{chen2021higher}
Chen J.P., Sau F., Higher-order hydrodynamics and equilibrium fluctuations of
 interacting particle systems, \textit{Markov Process. Related Fields}
 \textbf{27} (2021), 339--380, \href{https://arxiv.org/abs/2008.13403}{arXiv:2008.13403}.

\bibitem{corwin2016asepqj}
Corwin I., Shen H., Tsai L.-C., {${\rm ASEP}(q,j)$} converges to the {KPZ}
 equation, \href{https://doi.org/10.1214/17-AIHP829}{\textit{Ann. Inst. Henri Poincar\'e Probab. Stat.}} \textbf{54}
 (2018), 995--1012, \href{https://arxiv.org/abs/1602.01908}{arXiv:1602.01908}.

\bibitem{Dermoune2008SpectralGF}
Dermoune A., Heinrich P., Spectral gap for multicolor nearest-neighbor
 exclusion processes with site disorder, \href{https://doi.org/10.1007/s10955-008-9496-0}{\textit{J.~Stat. Phys.}} \textbf{131}
 (2008), 117--125.

\bibitem{franceschini2017stochastic}
Franceschini C., Giardin\`a C., Stochastic duality and orthogonal polynomials,
 \href{https://arxiv.org/abs/1701.09115}{arXiv:1701.09115}.

\bibitem{Giardin__2009}
Giardin\`a C., Kurchan J., Redig F., Vafayi K., Duality and hidden symmetries
 in interacting particle systems, \href{https://doi.org/10.1007/s10955-009-9716-2}{\textit{J.~Stat. Phys.}} \textbf{135} (2009),
 25--55, \href{https://arxiv.org/abs/0810.1202}{arXiv:0810.1202}.

\bibitem{https://doi.org/10.1111/j.1467-842X.1971.tb01239.x}
Griffiths R.C., Orthogonal polynomials on the multinomial distribution,
 \href{https://doi.org/10.1111/j.1467-842x.1971.tb01239.x}{\textit{Aust.~J. Stat.}} \textbf{13} (1971), 27--35.

\bibitem{Groenevelt_2018}
Groenevelt W., Orthogonal stochastic duality functions from {L}ie algebra
 representations, \href{https://doi.org/10.1007/s10955-018-2178-7}{\textit{J.~Stat. Phys.}} \textbf{174} (2019), 97--119,
 \href{https://arxiv.org/abs/1709.05997}{arXiv:1709.05997}.

\bibitem{iliev_2012}
Iliev P., A {L}ie-theoretic interpretation of multivariate hypergeometric
 polynomials, \href{https://doi.org/10.1112/S0010437X11007421}{\textit{Compos. Math.}} \textbf{148} (2012), 991--1002,
 \href{https://arxiv.org/abs/1101.1683}{arXiv:1101.1683}.

\bibitem{Kuan_2017}
Kuan J., A multi-species {${\rm ASEP}(q,j)$} and {$q$}-{TAZRP} with stochastic
 duality, \href{https://doi.org/10.1093/imrn/rnx034}{\textit{Int. Math. Res. Not.}} \textbf{2018} (2018), 5378--5416,
 \href{https://arxiv.org/abs/1605.00691}{arXiv:1605.00691}.

\bibitem{kuan2019stochastic}
Kuan J., Stochastic fusion of interacting particle systems and duality
 functions, \href{https://arxiv.org/abs/1908.02359}{arXiv:1908.02359}.

\bibitem{2018}
Redig F., Sau F., Factorized duality, stationary product measures and
 generating functions, \href{https://doi.org/10.1007/s10955-018-2090-1}{\textit{J.~Stat. Phys.}} \textbf{172} (2018), 980--1008,
 \href{https://arxiv.org/abs/1702.07237}{arXiv:1702.07237}.

\bibitem{zhou2020hydrodynamic}
Zhou Z., Hydrodynamic limit for a {$d$}-dimensional open symmetric exclusion
 process, \href{https://doi.org/10.1214/20-ecp350}{\textit{Electron. Commun. Probab.}} \textbf{25} (2020), 76, 8~pages,
 \href{https://arxiv.org/abs/2004.14279}{arXiv:2004.14279}.

\end{thebibliography}
\end{document}